\documentclass[12pt]{article}
\usepackage{epsfig}
\usepackage{amssymb}
\usepackage{amsmath}

\usepackage{graphicx}

\usepackage{setspace}
\usepackage{enumitem}
\setlist{nolistsep}

\usepackage{lscape}

\newcommand{\mathsfe}{\mathsf e}

\renewcommand{\O}{\mathcal{O}}

\newcommand{\PG}{\textup{PG}}

\newcommand{\T}{\mathcal{T}}
\renewcommand{\H}{\mathcal{H}}

\newcommand{\E}{\mathcal{E}}

\newcommand{\Label}{\label}

\newcommand{\sw}{w}
\renewcommand{\sb}{b}

\newtheorem{theorem}{Theorem}[section]

\newtheorem{lemma}[theorem]{Lemma}

\newtheorem{remark}[theorem]{Remark}

\newenvironment{proof}{\noindent{\bfseries Proof}\hspace{0.5em}}{ \null \hfill $\square$ \par}

\usepackage{fancyhdr}
\usepackage[usenames,dvipsnames]{color}

\begin{document}

\title{Characterising elliptic solids of $Q(4,q)$, $q$ even}

\author{S.G. Barwick, Alice M.W. Hui and Wen-Ai Jackson
\date{\today}
}

\maketitle



\begin{abstract}
Let $\E$ be a set of solids (hyperplanes) in $\PG(4,q)$, $q$ even, $q>2$, such that every point of $\PG(4,q)$ lies in either $0$, $\frac12q^3$ or $\frac12(q^3-q^2)$ solids of $\E$, and every plane of $\PG(4,q)$ lies in either $0$, $\frac12q$ or $q$ solids of $\E$. This article shows that $\E$ is either the set of solids that are disjoint from a hyperoval, or the set of solids that meet a non-singular quadric $Q(4,q)$ in an elliptic quadric.
\end{abstract}

\section{Introduction}

In this article, we denote  a non-singular quadric of $\PG(4,q)$  by  $Q(4,q)$, and refer to the hyperplanes of $\PG(4,q)$ as solids.
We prove the following characterisations, the second characterises  the solids of $\PG(4,q)$ meeting a non-singular quadric $Q(4,q)$ in an elliptic quadric.

\begin{theorem}\Label{mainthm-1} Let $\E$ be a set of solids   in $\PG(4,q)$, $q$ even, $q>2$, such that
 \begin{itemize}
\item[(I)] every point of $\PG(4,q)$ lies in either $0$, $\frac12q^3$ or $\frac12(q^3-q^2)$ solids of $\E$, and
\item[(II)] every plane of $\PG(4,q)$ lies in either $0$, $\frac12q$ or $q$ solids of $\E$.
\end{itemize}
Then
either $\E$ is the set of solids that are disjoint from a hyperoval, or $\E$ is the set of solids that meet a non-singular quadric $Q(4,q)$ in an elliptic quadric.
\end{theorem}

\begin{theorem}\Label{mainthm-2} Let $\E$ be a set of solids   in $\PG(4,q)$, $q$ even, $q>2$, satsifying conditions (I) and (II) of Theorem~\ref{mainthm-1}. If
 \begin{itemize}
\item[(III)] there exists a line of $\PG(4,q)$ that lies in exactly $\frac12 q(q+1)$ solids of $\E$,
\end{itemize}
then
 $\E$ is the set of solids that meet a non-singular quadric $Q(4,q)$ in an elliptic quadric.
\end{theorem}

For more details on non-singular quadrics, see \cite{HT}. Note that using techniques from \cite{HT}, we can show that  the set of solids that meet a non-singular quadric $Q(4,q)$ in an elliptic quadric satisfies conditions (I), (II) and (III).

There are several known characterisations of the  points, lines, planes and solids relating to $Q(4,q)$; these are surveyed in  \cite{BHJ}.

%
%
%

\section{Preliminary results}

 Let $\E$ be a set of solids in $\PG(4,q)$, $q$ even, $q>2$, such that
 \begin{itemize}
\item[(I)] every point of $\PG(4,q)$ lies in either $0$, $\frac12q^3$ or $\frac12(q^3-q^2)$ solids of $\E$, and \item[(II)] every plane of $\PG(4,q)$ lies in either $0$, $\frac12q$ or $q$ solids of $\E$.
\end{itemize}
By assumption (I), there are three types of points; a point is called a
\begin{itemize}
\item \emph{red point} if it lies in $0$ solids of $\E$,
\item \emph{white point} if it lies in $\frac12q^3$ solids of $\E$,
\item \emph{black point} if it lies in $\frac12(q^3-q^2)$ solids of $\E$.
\end{itemize}
Further, by assumption (I), a solid in $\E$ does not contain any red points, so we partition the solids of $\PG(4,q)$ into $\E$ and two further types of solids:
\begin{itemize}
\item
let $\T$ be the set of solids containing at least one red point,
\item
let $\H$ be the set of solids not in $\E$ that contain no red points.
\end{itemize}

In the following, we will show that there are two possibilities. In case A, the red points form a hyperoval $\mathcal O$ in a plane $\pi$, the white points are the points in $\pi\setminus\mathcal O$, and the black points are the points of $\PG(4,q)\setminus\pi$. Moreover, in this case, $\E$ is the set of solids that do not meet $\mathcal O$, $\T$ is the set of solids that meet $\mathcal O$, and $\H=\emptyset$.
In case B, the black points form a  non-singular quadric $Q(4,q)$ whose nucleus is the unique red point. Moreover, in this case, $\E$, $\H$, $\T$ are  the set of solids meeting $Q(4,q)$ in an elliptic quadric,  a hyperbolic quadric and    a conic cone respectively.

\subsection{Determining $|\E|$}

\begin{lemma}\Label{lem:A-size1}
Let $$\mathsfe = \dfrac{|\E|}{\frac12q^2},$$ then
 $\mathsfe$ is an integer and  $\mathsfe$ is congruent to either $0$ or $-1$ modulo $q$.

\end{lemma}

\begin{proof}
We first show that $\frac12q^2$ divides $|\E|$.
Let $r,w,b$ be the number of red, white and black points of $\PG(4,q)$ respectively.
Count in two ways the incident pairs $(P,\Pi)$ where $\Pi$ is a solid in $ \E$ and $P$ is a point of $\Pi$. We have
\begin{eqnarray}
r.0+w.\tfrac12 q^3 + b.\tfrac12(q^3-q^2)& =& |\E|(q^3+q^2+q+1). \label{eqn:OP}
\end{eqnarray}
As $\frac12q^2$ divides the left hand side and gcd$(q^2,q^3+q^2+q+1)=1$, it follows that $\frac12q^2$ divides $|\E|$, so $\mathsfe$ is an integer.

Now we show that $ \mathsfe$ is congruent to either 0 or $-1$ modulo $q$.
Count in two ways incident triples $(P, \Pi,\Sigma)$ where $\Pi,\Sigma$ are distinct solids of $\E$ and $P$ is a point in $\Pi\cap\Sigma$. We have
\begin{eqnarray}
r.0.0+ w.\tfrac12q^3(\tfrac12q^3-1)+b.\tfrac12(q^3-q^2).(\tfrac12(q^3-q^2)-1)&=&|\E|(|\E|-1)(q^2+q+1).\label{eqn:1}
\end{eqnarray}
Substituting (\ref{eqn:OP}) into (\ref{eqn:1}) to eliminate $b$ and then simplifying gives
\begin{eqnarray*}
wq&=&\mathsfe\left(\left[\mathsfe-(q-1)\right]
(q^2+q+1)-q(q^3-q^2-2)\right).\\
\end{eqnarray*}
Consider the right hand side of this equation. If the first factor $\mathsfe$ is even, then the second factor is odd, and so $\mathsfe$ is congruent to 0 modulo $q$. On the other hand, if the first factor $\mathsfe$ is odd, then the second factor is congruent to 0 modulo $q$, and so $\mathsfe$ is congruent to $-1$ modulo $q$.
\end{proof}

\begin{lemma}\Label{lem:E-size-s}
Each solid in $\E$ contains $s$ black points where  \begin{eqnarray}
s +q&=&(q^2-\mathsfe)(q^2+q+1).\label{eqn:3-pt-A-m}
\end{eqnarray}
\end{lemma}

\begin{proof}
Let $\Sigma\in\E$, let $ t $ be the number of white points in $\Sigma$ and let $ s $ be the number of black points in $\Sigma$. As $\Sigma$ does not contain any red points, we have
\begin{eqnarray}
 t + s &=&q^3+q^2+q+1.\label{eqn:3-pt-A}
\end{eqnarray}
Count in two ways incident pairs $(P,\Pi)$ where $\Pi\in\E$, $\Sigma\ne \Pi$ and $P$ is a point in $\Sigma\cap\Pi$. We have
\begin{eqnarray*}
t .(\tfrac12 q^3-1)+ s .(\tfrac12(q^3-q^2)-1)&=&(|\E|-1)(q^2+q+1).
\end{eqnarray*}
Simplifying using (\ref{eqn:3-pt-A})   gives (\ref{eqn:3-pt-A-m}) as required.
\end{proof}

\begin{lemma}\Label{lem:A-size1-strong}
If $\mathsfe \equiv 0 \pmod{q}$, then
\begin{enumerate}
\item $|\E|=\frac12 q^3 (q-1)$,
\item each solid in $\E$ contains $q^3+q^2$ black points and $q+1$ white points.
\end{enumerate}
\end{lemma}

\begin{proof}
Suppose
$\mathsfe \equiv0 \pmod{q}$, so $\mathsfe=kq$ for some integer $k$, and hence $|\E|=\tfrac12 kq^3$.
By Lemma~\ref{lem:E-size-s}, a solid in $\E$ contains $s$ black points, with $s$ given in (\ref{eqn:3-pt-A-m}).
 Substituting $\mathsfe=kq$ into (\ref{eqn:3-pt-A-m}) gives
\begin{eqnarray*}
 s +q&=&(q^2-kq)(q^2+q+1).
\end{eqnarray*}
So we have $q^2+q+1\bigm| s +q$ and $q\bigm| s +q$. As gcd$(q^2+q+1,q)=1$, it follows that $ s =xq(q^2+q+1)-q$ for some integer $x$. By (\ref{eqn:3-pt-A}), $0\leq s \le q^3+q^2+q+1$, hence $x=1$ and so $ s =q^2 (q + 1)$, $k=q-1$ and
$
|\E|=\frac12 q^3 (q-1).
$
\end{proof}

\begin{lemma}\Label{lem:E-size}
If $\mathsfe \equiv -1 \pmod{q}$, then \begin{enumerate}
\item $|\E|=\frac12 q^2 (q^2-1)$,
\item
each solid in $\E$ contains $q^2+1$ black points and $q^3+q$ white points.
\end{enumerate}
\end{lemma}

\begin{proof}
If
 $\mathsfe \equiv -1 \pmod{q}$,
then $\mathsfe=kq-1$ for some integer $k$, and hence $|\E|=\frac12 q^2(kq-1)$. We will show that $k=q$.
By Lemma~\ref{lem:E-size-s}, a solid in $\E$ contains $s$ black points, with $s$ given in (\ref{eqn:3-pt-A-m}).
Substituting $\mathsfe=kq-1$ into (\ref{eqn:3-pt-A-m}) gives
\begin{eqnarray}
 s +q&=&(q^2-kq+1)(q^2+q+1).\label{eqn:size-A2}
\end{eqnarray}
So $q^2+q+1\bigm| s +q$, that is, for some integer $x$, we have
\begin{equation}
 s +q=x(q^2+q+1).\label{eqn:size-A1}
\end{equation}
Comparing (\ref{eqn:size-A2}) and (\ref{eqn:size-A1}) we get $x-1=q(q-k)$, so $q\bigm|x-1$.
By   (\ref{eqn:3-pt-A}), $0\leq s \le q^3+q^2+q+1$, hence $0<x\leq q$. It follows that $x=1$ and $k=q$. This gives $s =q^2+1$ and $|\E|=\frac12 q^2(q^2-1)$, as required.
\end{proof}

\subsection{Case A: $\mathsfe \equiv 0 \pmod{q}$, so $|\E|=\frac12 q^3 (q-1)$}\Label{secA}

\begin{lemma}\Label{congA-rwb}
If $|\E|=\frac12 q^3 (q-1)$, then in $\PG(4,q)$, there are $q+2$ red points, $q^2-1$ white points and $q^4+q^3$ black points.
\end{lemma}

\begin{proof} Let $r$ be the number of red points, $w$ be the number of white points, and $b$ be the number of black points of $\PG(4,q)$.
Substituting $|\E|=\frac12 q^3 (q-1)$ into (\ref{eqn:OP}) gives
\begin{eqnarray}
\sw q + \sb (q-1)&=& q(q-1)(q^3+q^2+q+1). \label{eqn:4-rev}
\end{eqnarray}
Substituting $|\E|=\frac12 q^3 (q-1)$ into (\ref{eqn:1}) and dividing by $\frac14 q^2$ gives
\begin{eqnarray}
wq(q^3-2) + b(q-1)(q^3-q^2-2) &=& q(q-1)(q^3(q-1)-2)(q^2+q+1).\label{eqn:5-rev}
\end{eqnarray}
Multiplying (\ref{eqn:4-rev}) by $q^3-2$ and subtracting (\ref{eqn:5-rev}) gives
\[
q^2(q-1)b= q (q - 1)\left[( q^3 - 2)(q^3+q^2 + q+1)-(q^4-q^3-2)( q^2 + q + 1)\right]
\]
giving $\sb =q^4 + q^3 $. Using (\ref{eqn:4-rev}) we get $w=q^2-1$. As $b+w+r=q^4+q^3+q^2+q+1$, we have $r=q+2$.
\end{proof}

\begin{lemma}\Label{congA-red}
If $|\E|=\frac12 q^3 (q-1)$, then the $q+2$ red points of $\PG(4,q)$ lie in a plane.
\end{lemma}

\begin{proof}
By assumption (II), there are three types of planes in $\PG(4,q)$. Let $\Sigma\in\E$ and let $x,\ y,\ z$ denote the number of planes in $\Sigma$ that lie in $0,\ \frac12q, \ q$ solids of $\E$ respectively. Noting that $x=0$,
 we have
\begin{equation}
y+z=q^3+q^2+q+1.\label{eqn:yz}
\end{equation}
Counting in two ways the pairs $(\pi,\Pi)$ where $\pi$ is a plane in $\Sigma$, and $\Pi$ is a solid containing $\pi$, with $\Pi\in\E$ and $\Pi\neq\Sigma$ gives \begin{eqnarray}
y.(\tfrac12q-1)+z(q-1)&=&(|\E|-1).1.\nonumber 
\end{eqnarray}
Using $|\E|=\frac12 q^3 (q-1)$ and (\ref{eqn:yz}) gives $z=q+1$ and $y=q^3+q^2$.
Let $\alpha_1,\ldots,\alpha_{q+1}$ be the $z=q+1$ planes in $\Sigma$ which lie in $q$ solids of $\E$ and so in one solid not in $\E$. For $i=1,\ldots,q+1$, let $\Gamma_i$ denote the unique solid not in $\E$ that contains $\alpha_i$. Recall that the solids in $\E$ contain no red points. For any fixed $i$, the solids containing $\alpha_i$ cover all the points of $\PG(4,q)$, so the $r=q+2$ red points lie in $\Gamma_i$. In particular, as $\alpha_1,\alpha_2$ are distinct planes in $\Sigma$, $\pi=\Gamma_1\cap\Gamma_2$ is a plane which is not contained in $\Sigma$. The $r=q+2$ red points lie in $\Gamma_1$ and $\Gamma_2$, hence they lie in $\pi$.
 \end{proof}

\begin{lemma}\Label{A-red-noblack}
If $|\E|=\frac12 q^3 (q-1)$, then the plane $\pi$ containing the $q+2$ red points contains no black points.
\end{lemma}

\begin{proof}
By Lemma~\ref{congA-red}, there are $q+2$ red points which lie in a plane $\pi$.
Let $u$ be the number of white points in $\pi$ and $v$ the number of black points in $\pi$. As the number of red points in $\pi$ is $q+2$,
we have $u+v=(q^2+q+1)-(q+2)$, that is,
\begin{eqnarray}
u+v&=&q^2-1.  \label{eqn:wd-bd}
\end{eqnarray}
As solids in $\E$ contain no red points, $\pi$ does not lie in a solid of $\E$.
 Hence each solid in $\E$ meets $\pi$ in a line. We now count in two ways the incident pairs $(P,\Pi)$ where $\Pi$ is a solid in $\E$ and $P$ is a point in $\pi\cap\Pi$.
 Using assumption (I), we have
\begin{eqnarray*}
(q+2).0+u.(\tfrac12q^3)+v.(\tfrac12(q^3-q^2))&=&|\E|.(q+1)
\end{eqnarray*}
 Using $|\E|=\frac12 q^3 (q-1)$ and (\ref{eqn:wd-bd}) gives $v=0$.
\end{proof}

\begin{lemma}\Label{A-red-hyperoval}
If $|\E|=\frac12 q^3 (q-1)$, then the $q+2$ red points of $\PG(4,q)$ form a hyperoval.
\end{lemma}

\begin{proof} By Lemma~\ref{congA-red}, there are $q+2$ red points of $PG(4,q)$ and they lie in a plane $\pi$. Further, $\pi$ contains 0 black points and $q^2-1$ white points Lemma~\ref{A-red-noblack}. Let $\ell$ be a line of $\pi$ which contains $x\ge1$ red points, so $\ell$ contains $q+1-x$ white points.
As $\ell$ contains a red point, $\ell$ does not lie in a solid of $\E$, so each solid of $\E$ meets $\ell$ in a white point.
We count in two ways incident pairs $(P,\Pi)$ where $\Pi\in\E$ and $P$ is a white point in $\ell\cap\Pi$.  Using assumption (I), we have
 $$(q+1-x). \tfrac12q^3=\tfrac12q^3(q-1).1$$
and so $x=2$. Hence the $q+2$ points in $\pi$ are no three collinear, so they form a hyperoval.
\end{proof}

\begin{lemma}\Label{A-red-hyperoval-B}
If $|\E|=\frac12 q^3 (q-1)$, then $\E$ is the set of solids that are disjoint from a hyperoval.
\end{lemma}

\begin{proof}
By  Lemma~\ref{A-red-hyperoval}, there are $q+2$ red points of $PG(4,q)$ and they form a hyperoval $\O$ in a plane $\pi$.
Let $\Sigma$ be a solid in $\E$. As $\Sigma$ contains no red points, it does not contain $\pi$ and $\Sigma\cap\pi$ is a line that is
disjoint from $\O$.

Conversely we count the solids in $\PG(4,q)$ that are
disjoint from $\O$, that is, we count the number of solids of $\PG(4,q)$ that meet $\pi$ in a line containing 0 red points. The number of    lines of $\pi$ that are external to $\O$ is $q^2+q+1-\frac12 (q+2)(q+1)=\frac12(q^2-q)$. Through each of these external lines, there are $q^2+q+1$ solids of $\PG(4,q)$, $q+1$ of which contain $\pi$.
Hence the number of solids of $\PG(4,q)$ that meet $\pi$ in a line containing 0 red points is  $\frac12(q^2-q).[(q^2+1+1)-(q+1)] = \frac12(q^2-q).q^2$   which is equal to $|\E|$.
Hence $\E$ is the set of solids that are disjoint from $\O$.
\end{proof}

\subsection{Case B: $\mathsfe \equiv -1 \pmod{q}$ and so $|\E|=\frac12 q^2 (q^2-1)$}\Label{secB}

We now count the number of points of each colour in $\PG(4,q)$.

\begin{lemma} \label{item:pt-size}
If $|\E|=\frac12 q^2 (q^2-1)$, then in $\PG(4,q)$, there are $q^3+q^2+q+1$ black points, $q^4-1$ white points and one red point.
\end{lemma}

\begin{proof}
In $\PG(4,q)$, let $r,w,b$ be the number of red, white and black points respectively.
Count in two ways incident pairs
 $(P,\Pi)$ where $\Pi\in\E$ and $P$ is black point in $\Pi$.
 By Lemma~\ref{lem:E-size},
 $\Pi$ contains $q^2+1$ black points, so using assumption (I), we have
\begin{eqnarray*}
b. \tfrac12(q^3-q^2)=\tfrac12q^2(q^2-1).(q^2+1)
\end{eqnarray*}
giving $b=q^3+q^2+q+1$. Substituting this and $|\E|=\frac12q^2(q^2-1)$ into (\ref{eqn:OP}) gives $w=q^4-1$. As $b+w+r=q^4+q^3+q^2+q+1$, we also have $r=1$.
\end{proof}

Recall that we partition the solids of $\PG(4,q)$ into three sets: $\E$, $\T$ (the solids containing the red point), and $\H$ (the solids not in $\E$ which do not contain the red point). In Lemma~\ref{lem:E-size} we counted the solids in $\E$. We now do the same for $\T$ and $\H$.
\begin{lemma}\Label{lem:HT-size}
If $|\E|=\frac12 q^2 (q^2-1)$, then \begin{enumerate}
\item Each solid in $\H$ contains $(q+1)^2$ black points, and $|\H|=\frac12q^2(q^2+1)$;
\item Each solid in $\T$ contains $q^2+q+1$ black points, and $|\T|=q^3+q^2+q+1$.
\end{enumerate}
\end{lemma}

%
%
%
%

\begin{proof} By Lemma~\ref{item:pt-size} there is a unique red point, so we have $|\T|=q^3+q^2+q+1$ (the number of solids through a point). Using $|\E|=\frac12q^2(q^2-1)$ and $|\E|+|\H|+|\T|=q^4+q^3+q^2+q+1$ gives $|\H|=\frac12q^2(q^2+1)$ as required.

Let $\Sigma\in\H$, let $s $ be the number of black points in $\Sigma$ and let $t$ be the number of white points in $\Sigma$. As solids in $\H$ do not contain a red point, we have
\begin{eqnarray}
 t + s &=&q^3+q^2+q+1.\label{eqn:Ept}
\end{eqnarray}
Count in two ways incident pairs $(P,\Pi)$ where $\Pi\in\E$, and $P$ is a point in $\Pi\cap\Sigma$. Using assumption (I), we have
\begin{eqnarray*}
 t .\tfrac12 q^3+ s .\tfrac12(q^3-q^2)&=&|\E|(q^2+q+1).
\end{eqnarray*}
Simplifying using (\ref{eqn:Ept}) gives $ s =(q+1)^2$
proving part 1.

For part 2, let $\Sigma\in\T$, let $ s $ be the number of black points in $\Sigma$ and let $ t $ be the number of white points in $\Sigma$. As solids in $\T$ contain the unique red point, we have
\begin{eqnarray}
 t + s &=&q^3+q^2+q.\label{eqn-bing}
\end{eqnarray}
 Count in two ways incident pairs $(P,\Pi)$ where $\Pi\in\E$, and $P$ is a point in $\Pi\cap\Sigma$, As $\Pi\in\E$, it does not contain the red point, so using assumption (I), we have
 \begin{eqnarray*}
 t .\tfrac12 q^3+ s .\tfrac12(q^3-q^2)&=&|\E|(q^2+q+1).
\end{eqnarray*}
Simplifying using (\ref{eqn-bing}) gives $ s =q^2+q+1$ as required.
\end{proof}

\begin{lemma} \label{lem:pl-thru-N}
If $|\E|=\frac12 q^2 (q^2-1)$, then a plane containing the unique red point of $\PG(4,q)$ contains exactly $q+1$ black points.
\end{lemma}

\begin{proof}
Let $\pi$ be a plane through the red point, and let $x$ be the number of black points in $\pi$. By Lemma~\ref{item:pt-size}, there are $q^3+q^2+q+1-x$ black points not in $\pi$. Also, the $q+1$ solids containing $\pi$ lie in $\T$ and partition the remaining black points. Using this we count in two ways incident pairs $(P,\Pi)$ where $\Pi$ is a solid in $\T$ containing $\pi$ and $P$ is a black point in $\Pi$. By Lemma \ref{lem:HT-size}, we have
\begin{eqnarray*}
x.(q+1)+(q^3+q^2+q+1-x).1&=&(q+1).(q^2+q+1)
\end{eqnarray*}
and so $x=q+1$ as required.
\end{proof}
\begin{lemma} \label{lem:tgt-ln-N}
If $|\E|=\frac12 q^2 (q^2-1)$, then every line through the unique red point of $\PG(4,q)$ contains exactly one black point.
\end{lemma}
\begin{proof}
Let $\ell$ be a line through the red point and suppose $\ell$ contains $y$ black points. By Lemma~\ref{lem:pl-thru-N}, each plane containing $\ell$ contains $q+1-y$ further black points. Counting the total number of black points of $\PG(4,q)$ in two ways: by looking at planes through $\ell$ and by using Lemma~\ref{item:pt-size} gives
$$y+(q^2+q+1)(q+1-y)=q^3+q^2+q+1.$$ Hence $y=1$ as required.
\end{proof}

We now consider assumption (II), that a plane of $\PG(4,q)$ lies in $0$, $\frac12q$ or $q$ solids in $\E$, and determine properties of each class of planes.

\begin{lemma}\Label{lem:plane-char}
Suppose $|\E|=\frac12 q^2 (q^2-1)$, and let $\pi$ be a plane of $\PG(4,q)$ not containing the red point.
Then one of the following holds. 
\begin{enumerate}
\item  $\pi$ lies in $0$ solids of $\E$, $q$ solids of $\H$ and $\pi$ contains $2q+1$ black points.
\item $\pi$ lies in $\frac12q$ solids of $\E$,  $\frac12q$ solids of $\H$ and  $\pi$ contains $q+1$ black points.
\item  $\pi$ lies in $q$ solids of $\E$,  $0$ solids of $\H$ and  $\pi$ contains $1$ black point.
\end{enumerate}
\end{lemma}

\begin{proof}
Let $\pi$ be a plane not containing the red point, and suppose $\pi$ contains $x$ black points. Let $e $ be the number of solids in $\E$ through $\pi$ and let $h $ be the number of solids in $\H$ through $\pi$. As each solid in $\T$ contains the red point, there is exactly one solid $\Sigma$ in $\T$ through $\pi$, so $e +h =q$. We count the black points in two ways, firstly using Lemma~\ref{item:pt-size}, and secondly via solids of $\T,\E,\H$ containing $\pi$ using Lemma \ref{lem:HT-size}. We have
\begin{eqnarray*}
q^3+q^2+q+1&=&x\,+\,1.(q^2+q+1-x)+e .(q^2+1-x)+h .(q^2+2q+1-x).
\end{eqnarray*}
Simplifying gives $x=2q+1-2e$. By assumption (II), there are three possibilities for $e$, namely $0$, $q/2$, and $q+1$. If  $e =0$, then  $x=2q+1$. If $e =q/2$, then $x=q+1$. If $e =q$, then $x=1$.
\end{proof}

Hence there are three different types of planes of $\PG(4,q)$. We shall call a plane a \emph{1-plane, $(q+1)$-plane} or \emph{$(2q+1)$-plane} according to the number of black points it contains.

\begin{lemma}\Label{lem:ovoid}
If $|\E|=\frac12 q^2 (q^2-1)$, and $\Pi\in\E$, then the black points in $\Pi$ form
 an ovoid.
\end{lemma}

\begin{proof}
Let $\Pi\in\E$ and suppose there exists a line $\ell$ in $\Pi$ containing $3+x$ black points for some $x\ge 0$.
Let $\pi$ be a plane of $\Pi$ containing $\ell$. There are $q+1$ solids of $\PG(4,q)$ containing $\pi$, one contains the red point and one is $\Pi$.
As $\pi$ contains at least three black points, by Lemma~\ref{lem:plane-char}, $\pi$ is not a 1-plane. As $\pi$ lies in a solid of $\E$, by Lemma~\ref{lem:plane-char}, $\pi$ is a $(q+1)$-plane.
 Counting the black points in $\Pi$ in two ways,  firstly using Lemma~\ref{lem:E-size} and secondly by counting the planes of $\Pi$ through $\ell$ gives
$$q^2+1=(3+x)+(q+1)(q+1-(3+x)),$$
giving a negative value of $x$, a contradiction.
Thus the lines in $\Pi$ contain at most two black points. Hence the $q^2+1$ black points in $\Pi$ are no three collinear, and form
 an ovoid.
\end{proof}

\begin{lemma}\Label{lem:pl-thro-ell-inH} Suppose
$|\E|=\frac12 q^2 (q^2-1)$.
Let $\Pi\in\H$ and let $\ell$ be a line in $\Pi$ which contains $y$ black points, so $0\leq y\leq q+1$.
Then the number of $(q+1)$-planes and $(2q+1)$-planes of $\Pi$ containing $\ell$ is $q+1-y$ and $y$ respectively.
\end{lemma}
%

\begin{proof}
Let $\Pi\in\H$ and let $\ell$ be a line in $\Pi$ which contains $y$ black points, $0\leq y\leq q+1$.
Let $n$ be the number of $(q+1)$-planes of $\Pi$ through $\ell$ and let
 $m$ be the number of $(2q+1)$-planes of $\Pi$ through $\ell$. It follows from Lemma \ref{lem:plane-char} that there are no $1$-planes in $\Pi$, so
\begin{eqnarray}
n+m&=&q+1.\label{eqn-bong}
\end{eqnarray}
By Lemma~\ref{lem:HT-size}, $\Pi$ contains $(q+1)^2$ black points, so counting the black points of $\Pi$ gives
\begin{eqnarray*}
y+n.(q+1-y)+m.(2q+1-y)&=&(q+1)^2.
\end{eqnarray*}
Simplifying using (\ref{eqn-bong}) gives $m=y$ and $n=q+1-y$ as required.
\end{proof}

\begin{lemma}\Label{lem:hyp}
Suppose $|\E|=\frac12 q^2 (q^2-1)$. If $\Pi\in\H$, and $q>2$, then the black points in $\Pi$ form a hyperbolic quadric.
\end{lemma}

\begin{proof}
Let $\Pi\in\H$. We suppose $\ell$ is a line of $\Pi$ which contains $y$ black points with $3\leq y\leq q$, and work to a contradiction.
By Lemma~\ref{lem:pl-thro-ell-inH}, the number of $(q+1)$-planes containing $\ell$ is $q+1-y\geq 1$. So there is a $(q+1)$-plane $\pi$ in $\Pi$ containing $\ell$. By Lemma~\ref{lem:plane-char}, $\pi$ lies in a solid of $\E$. Hence by Lemma~\ref{lem:ovoid}, the black points in $\pi$ form an oval. As $\ell$ is a line of $\pi$, this contradicts $3\leq y\leq q$. Hence every line of $\Pi$ contains either $0$, $1$, $2$ or $q+1$ black points. By Lemma~\ref{item:pt-size}, there are $q^2+2q+1$ black points in $\Pi$. It follows from a characterisation of sets of type $(0,1,2,q+1)$ due to Tallini \cite{tallini}, that if $q>2$, the black points in $\Pi$ form a quadric containing $q^2+2q+1$ points, that is, a hyperbolic quadric.
\end{proof}
%
%
%

%
%

\begin{lemma}\Label{lem:quadcone}
Suppose $|\E|=\frac12 q^2 (q^2-1)$. If $\Pi\in\T$, and $q>2$, then each line in $\Pi$ contains $0$, $1$, $2$ or $q+1$ black points.
 \end{lemma}

\begin{proof}
Let $\Pi\in\T$ and $\ell$ a line of $\Pi$. Recall from Lemma~\ref{item:pt-size} that there is a unique red point which lies in each solid of $\T$. If $\ell$ contains the red point, then by Lemma~\ref{lem:tgt-ln-N}, $\ell$ contains exactly one black point. If $\ell$ does not contain the red point, then $\ell$ lies in a plane $\pi$ not through the red point. By Lemma~\ref{lem:plane-char}, $\pi$ lies in a solid of $\H$ or $\E$. Hence by Lemma~\ref{lem:ovoid} or Lemma~\ref{lem:hyp}, each line in $\pi$ contains
$0$, $1$, $2$ or $q+1$ black points.
Thus every line in $\Pi$ contains $0$, $1$, $2$ or $q+1$ black points as required.
\end{proof}

\section{Proof of main results}

We now complete the proofs of the main results. 

{\bfseries Theorem~\ref{mainthm-1} } \emph{Let $\E$ be a set of solids   in $\PG(4,q)$, $q$ even, $q>2$, such that
 \begin{itemize}
\item[(I)] every point of $\PG(4,q)$ lies in either $0$, $\frac12q^3$ or $\frac12(q^3-q^2)$ solids of $\E$, and
\item[(II)] every plane of $\PG(4,q)$ lies in either $0$, $\frac12q$ or $q$ solids of $\E$.
\end{itemize}
Then
either $\E$ is the set of solids that are disjoint from a hyperoval, or $\E$ is the set of solids that meet a non-singular quadric $Q(4,q)$ in an elliptic quadric.}

\begin{proof}
Let $\E$ be a set of solids in $\PG(4,q)$, $q$ even, $q>2$, satsifying satisfying Assumptions (I) and (II).
By Lemma~\ref{lem:A-size1}, $\mathsfe$ is congruent to $0$ or $-1$ modulo $q$. We consider these two cases separately.

\emph{Case A:\ }If
$\mathsfe$ is congruent to $0$ modulo $q$, then by Lemma~\ref{lem:A-size1-strong}, $|\E|=\frac12 q^3(q-1)$. By Lemma~\ref{A-red-hyperoval}, the red points form a hyperoval $\mathcal O$. Further, by    Lemma~\ref{A-red-hyperoval-B}, $\E$ is the set of solids that are disjoint from  $\mathcal O$.

\emph{Case B:\ } If
$\mathsfe$ is congruent to $-1$ modulo $q$, then by Lemma~\ref{lem:E-size}, $|\E|=\frac12 q^2 (q^2-1)$.
Let $\mathcal B$ be the set of black points. Let $\ell$ be a line of $\PG(4,q)$ and let $\Sigma$ be a solid containing $\ell$, so $\Sigma$ lies in one of $\E,\H,\T$. Hence
by one of Lemmas~\ref{lem:ovoid},~\ref{lem:hyp},~\ref{lem:quadcone}, $\ell$ contains
 $0$, $1$, $2$ or $q+1$ black points.
By Lemma~\ref{item:pt-size}, $|\mathcal B|=q^3+q^2+q+1$. It follows from a characterisation of sets of type $(0,1,2,q+1)$ due to Tallini \cite{tallini} that as $q>2$,   $\mathcal B$ is the set of points of either a  non-singular quadric  or a solid of $\PG(4,q)$.
By Lemmas~\ref{lem:ovoid}~and~\ref{lem:hyp}, the black points do not lie in a single solid, so the
  latter case does not occur. By Lemma \ref{lem:ovoid}, $\E$ is a set of solids of $\PG(4,q)$ that meet a  non-singular quadric  in an ovoid, which is necessarily an elliptic quadric.
\end{proof}

 \begin{remark}
 \begin{itemize}
 \item  In  case A, it follows that $\T$ is the set of  solids  that meet  the  hyperoval $\mathcal O$, and $\H=\emptyset.$
 Furthermore, if $\pi$ is the plane containing   $\mathcal O$, then the white points are the  points in $\pi\setminus\mathcal O$, and the black points are the points of $\PG(4,q)\setminus\pi$.
\item In   case B, it follows that $\T$ is the set of solids that meet the non-singular quadric in a quadratic cone, and $\H$ is the set of solids that meet the non-singular  quadric in a hyperbolic quadric.
\end{itemize}
\end{remark}

\begin{lemma}\label{lem:new}
\begin{enumerate}
\item Let $\E$ be the set of solids of $\PG(4,q)$ that are disjoint from a fixed hyperoval. Then a line of $\PG(4,q)$ lies in either $0$, $\frac12 q(q-1)$, $\frac12 q^2$, or $q^2$  solids of $\E$.
\item  Let $\E$ be the set of solids of $\PG(4,q)$ that meet a non-singular quadric $Q(4,q)$ in an elliptic quadric. Then a line of $\PG(4,q)$ lies in either $0$, $\frac12 q(q-1)$, $\frac12q^2$,  or  $\frac12 q(q+1)$ solids of $\E$.
\end{enumerate}
\end{lemma}

\begin{proof}
For part 1, let $\E$ be the set of solids of $\PG(4,q)$ that are disjoint from a fixed hyperoval $\O$ lying in a plane $\pi$. 
If $\ell$ is any line that contains one or two points of $\O$, then $\ell$ lies in $0$ solids of $\E$. Suppose $\ell$ is a line of $\pi$ that does not meet $\O$. Then $\ell$ lies in $q+1$ solids that contain $\pi$, and $q^2$ solids that meet $\pi$ in $\ell$. That is, $\ell$ lies in $q^2$ solids of $\E$. 
Suppose $\ell$ is not a line of $\pi$, but $\ell$   meets $\pi$ in a point not in $\O$. Then $\ell$ lies in $ 1$ solid that contains $\pi$, 
$\frac12(q+2)q$ solids that meet $\pi$ in a bisecant of $\O$, and $\frac12 q^2$ solids that  
 meet $\pi$ in an exterior line of $\O$. That is, $\ell$ lies in $\frac12 q^2$ solids of $\E$. Finally, suppose $\ell$ is any line that does not meet $\pi$, so each solid through $\ell$ meets $\pi $ in a line.  So $\frac12 (q+2)(q+1)$ solids through $\ell$ meet $\O$ in a secant line, and  $\frac12 q(q-1) $   solids through $\ell$ meet $\O$ in an exterior line. That is, $\ell$ lies in $\frac12 q(q-1) $  solids of $\E$. 
Part 2 is straightforward counting in $Q(4,q)$. 
\end{proof}

{\bfseries Theorem~\ref{mainthm-2}}\ \emph{ Let $\E$ be a set of solids   in $\PG(4,q)$, $q$ even, $q>2$, satsifying conditions (I) and (II) of Theorem~\ref{mainthm-1}. If
 \begin{itemize}
\item[(III)] there exists a line of $\PG(4,q)$ that lies in exactly $\frac12 q(q+1)$ solids of $\E$,
\end{itemize}
then
 $\E$ is the set of solids that meet a non-singular quadric $Q(4,q)$ in an elliptic quadric.}
 
 \begin{proof} As $\E$ satisfies conditions (I) and (II), 
 by Theorem~\ref{mainthm-1}, $\E$ is either the set of solids that are disjoint from a hyperoval, or $\E$ is the set of solids that meet a non-singular quadric $Q(4,q)$ in an elliptic quadric. If $\ell$ is a line that lies in   exactly  $\frac12 q(q+1)$ solids of $\E$, then by Lemma~\ref{lem:new}, $\E$ is not  the set of solids that are disjoint from a hyperoval, hence $\E$ is the set of solids that meet a non-singular quadric $Q(4,q)$ in an elliptic quadric. 
 \end{proof}

\bigskip\bigskip

{\bfseries Author details}

S.G. Barwick. School of Mathematical Sciences, University of Adelaide, Adelaide, 5005, Australia.
susan.barwick@adelaide.edu.au

A.M.W. Hui. Statistics Program, BNU-HKBU United International College, Zhuhai, China.
alicemwhui@uic.edu.hk, huimanwa@gmail.com

W.-A. Jackson. School of Mathematical Sciences, University of Adelaide, Adelaide, 5005, Australia.
wen.jackson@adelaide.edu.au

A.M.W. Hui acknowledges the support of the Young Scientists Fund (Grant No. 11701035) of the National Natural Science Foundation of China.

\end{document}